\documentclass[12pt,a4paper]{article}
\usepackage[utf8]{inputenc}
\usepackage[T1]{fontenc}
\usepackage{amsmath}
\usepackage{amssymb}
\usepackage{amsthm}
\usepackage{graphicx}
\usepackage{mathrsfs}
\usepackage{float}
\newtheorem{proposition}{Proposition}[section]

\newtheorem{theorem}{Theorem}[section]
\newtheorem{lemma}{Lemma}[section]

\title{Line Graph Characterization of Cyclic Subgroup Graph}	
\author{Siddharth Malviy, Vipul Kakkar, Swapnil Srivastava}

\date{}
\begin{document}
	\maketitle
	\noindent \textbf{{Abstract.}}
	The cyclic subgroup graph ${\Gamma(G)}$ of a group $G$ is the simple undirected graph with cyclic subgroups as a vertex set and two distinct vertices $H_1$ and $H_2$ are adjacent if and only if $H_1 \leq H_2$ and there does not exist any cyclic subgroup $K$ such that $H_1 < K < H_2$. In this paper, we classify all the finite groups $G$ such that $\Gamma(G)$ is the line graph of some graph. \\
    
	\noindent \textbf{{Keywords.}}  Cyclic subgroup graph, line graph \\
	
	\noindent \textbf{2020 MSC.} 05C25, 05C76\\
	
	\section{Introduction}
The study of the interaction between algebraic structures and graph theory is a fundamental research area in algebraic graph theory. Various researchers have explored different graphs associated with groups due to their significant application in different mathematical disciplines. A detailed study on different graphs defined on groups has been done by PJ Cameron \cite{1}. Several graphs have been defined in connection with groups such as commuting graph, power graph, cayley graph and many more. Among them is the subgroup graph $L(G)$ of a group $G$. A subgroup graph is a graph that shows the relationship between subgroups of a group. Let $S(G)$ denote the set of subgroups of a finite group $G$, then the subgroup graph of group $G$ is a graph whose vertices is $S(G)$ such that two vertices $H_1, H_2 \in S(G)$ are adjacent if either $H_1 \leq H_2$ or $H_2 \leq H_1$ and there is no subgroup $K \leq G$ such that $H_1 < K < H_2.$  \\
In this paper, we mainly focus on a special subgraph of $L(G)$, namely the cyclic subgroup graph $\Gamma(G)$ of $G$. Its vertex set is $C(G)$, the collection of all cyclic subgroups of $G$ and two vertices  $H_1$ and $H_2$ are adjacent if and only if $H_1 \leq H_2$ and there does not exist any cyclic subgroup $K$ such that $H_1 < K < H_2$. Marius \cite{2} studied on the number of edges of cyclic subgroup graphs of finite order groups. K. Sharma and A. Reddy \cite{3} explored the various properties of these graphs and characterized when the cyclic subgroup graph is bipartite, connected, complete and regular. I. M. Richards \cite{4} have discussed on the number of cyclic subgroups of a finite groups.\\
The line graph $L(\Gamma)$ of graph $\Gamma$ is the graph whose vertex set consists
of all edges of $\Gamma$; two vertices of $L(\Gamma)$ are adjacent if and only if they are
incident in $\Gamma$. Several studies have been focused on classifying groups whose power graph, commuting graph and other graphs are line graphs (see  \cite{5}, \cite{6} and \cite{7}). \\
In this paper, we aim to classify all finite groups $G$ such that $\Gamma$ is a line graph. This characterization helps in understanding the cyclic subgroups interaction within finite groups through a graphical representation.

\section{Preliminaries}
	The vertex set $V(\Gamma)$ and the edge set $E(\Gamma)\subseteq V(\Gamma)\times V(\Gamma)$ form an ordered pair that constitutes a graph $\Gamma$. If $\{u, v\}\in E(\Gamma)$, then two vertices, $u$ and $v$ are adjacent. If a graph has no loops or multiple edges, then it is referred to as a {simple graph}. In this study, we just take into consideration simple graphs.  A graph $\Gamma'$ such that $V(\Gamma')\subseteq V(\Gamma)$ and $E(\Gamma')\subseteq E(\Gamma)$ is called a subgraph  of a graph $\Gamma$.  
	
	Suppose that $X\subseteq V(\Gamma)$. Then the subgraph $\Gamma'$ induced by the set $X$ is a graph such that $V(\Gamma')=X$ and $u,v\in X$ are adjacent if and only if they are adjacent in $\Gamma$. A graph $\Gamma$ is considered complete if every pair of vertices is adjacent to one another.  The graph $\overline{\Gamma}$ such that $V(\Gamma)= V(\overline{\Gamma})$ and two vertices $u$ and $v$ are adjacent in $\overline{\Gamma}$ if and only if $u$ is not adjacent to $v$ in $\Gamma$ is the complement of a graph $\Gamma$. Throughout this paper, $\mathbb{Z}_n$ denotes the cyclic group of order $n$.  If a graph is a complete graph, then it is the line graph and complement of line graph of some graph. The cyclic subgroup graph $\Gamma(G)$ of a group $G$ is complete if and only if either $G$ is trivial or $G \cong \mathbb{Z}_p$, where $p$ is a prime.
	
	\noindent A characterization of line graph is described in the next lemma,  which is helpful in the sequel.
	
	\begin{lemma}{\rm \cite{line}}{\label{lg}}
		A graph $\Gamma$ is the line graph of some graph if and only if none of the nine graphs in $\mathrm{Figure \; \ref{figure 1}}$  is an induced subgraph of $\Gamma$.
	\end{lemma}
	\begin{figure}[H]
		\centering
		\includegraphics[scale=.7]{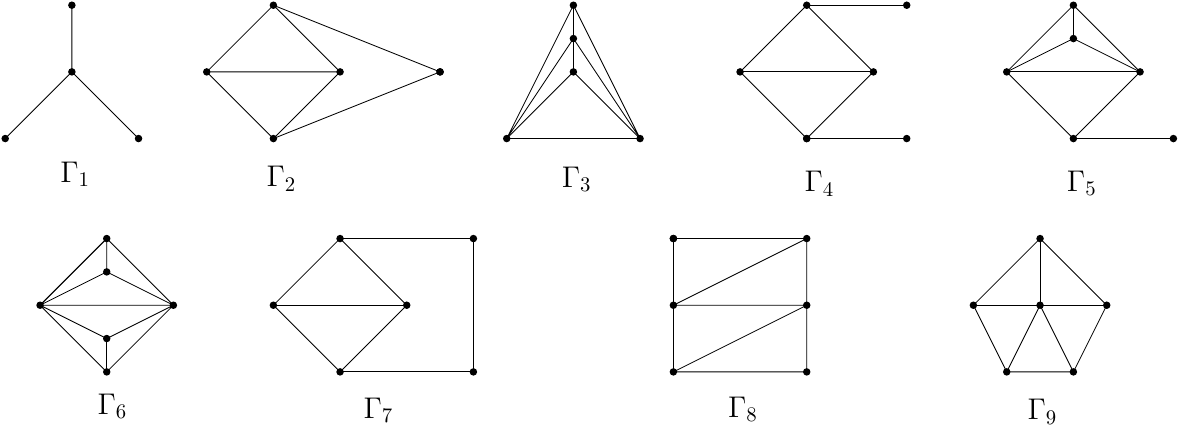}
		\caption{Forbidden induced subgraphs of line graphs.}
		\label{figure 1}
	\end{figure}
	
	\section{Line graph characterization of $\Gamma(G)$} 
	
	Throughout this section the graph $\Gamma$ or $\Gamma(G)$ will denote the line graph of a non-trivial finite group $G$. The identity of the group is denoted by $e$. All the finite groups $G$ such that $\Gamma(G)$ is a line graph of some graphs are classified in this section. 
	
\begin{theorem} \label{thm distinct}
     Let three distinct primes divide the order of the group $G$. Then the cyclic subgroup graph $\Gamma$ is not a line graph.
     \end{theorem}

\begin{proof}
    Let $p_1, p_2$ and $p_3$ be the three distinct which primes divide the order of the group $G$. Then by Cauchy's theorem  $G$ has elements $a,b$ and $c$ of order $p_1, p_2$ and $p_3$ respectively.\\ The subgroups $\langle a \rangle, \langle b \rangle$ and $\langle c \rangle$ generated by these elements are cyclic subgroups of $G$ of order $p_1, p_2$ and $p_3$ respectively. The subgraph of $\Gamma$ induced by the set $\{e, \langle a \rangle, \langle b \rangle, \langle c \rangle\}$ is isomorphic to $\Gamma_1$ of Figure \ref{figure 1}. By Lemma \ref{lg}, $\Gamma$ is not a line graph. 
\end{proof}

By Theorem \ref{thm distinct}, we will focus on the group $G$ such that the order $\mid G \mid$ of the group $G$ is $p^\alpha q^\beta$, where $p$ and $q$ are distinct primes and $\alpha, \beta \geq 0$.

\begin{theorem}\label{lcg}
    Let $G$ be a finite cyclic group of order $p^\alpha q^\beta$, where $p$ and $q$ are distinct primes and $\alpha, \beta \geq 0$. Then the cyclic subgroup graph $\Gamma$ is line graph if and only if $G$ is cyclic group of prime power order or $G\cong \mathbb{Z}_{pq}$.
\end{theorem}

\begin{proof}

If either $G$ is cyclic group of prime power order or $G \cong \mathbb{Z}_{pq}$, then one can check that the cyclic subgroup graph $\Gamma$ of the group $G$ is a line graph. \\

For the converse, we will consider the different cases depending on the order of the $G$. First, suppose that $|G|=p^\alpha q^\beta$, where $\alpha>1$ and $\beta \geq 1$. Then $G$ contains three subgroups $K_1$, $K_2$ and $K_3$ of order $p$, $p^2$ and $pq$ respectively. One can note that the subgraph of $\Gamma$ induced by the set $\{\{e\},K_1,K_2,K_3\}$ is isomorphic to $\Gamma_1$ of Figure \ref{figure 1}. Therefore, by Lemma \ref{lg}, $\Gamma$ is not a line graph. Similarly, if $|G|=p^\alpha q^\beta$, where $\alpha\geq 1$ and $\beta > 1$, then $\Gamma$ is not a line graph. 

\vspace{0.2 cm}

Therefore, we are left with the two cases. The group $G$ is either the cyclic group of prime power order or the cyclic group of order $pq$. In both the cases, the cyclic subgroup graph $\Gamma$ is a line graph. 
\end{proof}

\begin{theorem} \label{thm 1}
    Let $G$ be a finite non-cyclic abelian group of order $p^\alpha q^\beta$, where $p$ and $q$ are distinct primes and $\alpha, \beta \geq 0$. Then $\Gamma$ is not a line graph.  
\end{theorem}
\begin{proof}

Since $G$ is a non-cyclic abelian group, $G$ contains a subgroup $L$ isomorphic to $\mathbb{Z}_t \times \mathbb{Z}_t$ for some prime $t\in \{p,q\}$. The subgroup $L$ of $G$ contains three distinct subgroups $K_1, K_2$ and $K_3$ of order $t$. The induced subgraph $\{\{e\}, K_1, K_2, K_3\}$ is isomorphic to $\Gamma_1$ of Figure \ref{figure 1}. Hence, by Lemma \ref{lg}, $\Gamma$ is not a line graph.
\end{proof}

\begin{proposition} \label{prop 1}
    Let $G$ be a non-abelian group of order $pq$, where $p$ and $q$ are distinct primes. Then the cyclic subgroup graph $\Gamma$ of the group $G$ is not a line graph.
\end{proposition}
\begin{proof}
    Assume that $p<q$. Since $G$ is non-abelian, there exist at least three distinct Sylow $p$-subgroups of $G$. Then these three Sylow $p$-subgroups together with the trivial subgroups will make the structure $\Gamma_1$ of the Figure \ref{figure 1}. Therefore, by Lemma \ref{lg}, $\Gamma$ is not a line graph.
\end{proof}

\begin{theorem}
    Let $G$ be a group of order $p^\alpha q^\beta$, where $p$ and $q$ are distinct primes and $\alpha, \beta \geq 0$. Then the cyclic subgroup graph $\Gamma$ of the group $G$ is a line graph if and only if $G$ is either a cyclic group of prime power order or a cyclic group of order $pq$.
\end{theorem}

\begin{proof}
    Let $G$ be a cyclic group of prime power order or a cyclic group of order $pq$, then by Theorem \ref{lcg}, $\Gamma$ is a line graph.\\
    Conversely, suppose $G$ is non-trivial group of order $p^\alpha q^\beta$, where $p$ and $q$ are distinct primes and $\alpha, \beta \geq 0$ such that $\Gamma$ is a line graph. First, suppose that $\alpha >0$ and $\beta>0$. Let $P$ and $Q$ be a Sylow $p$-subgroup and Sylow $q$-subgroup of $G$ respectively. Now, suppose that one of $P$ and $Q$ (say $P$) is non-cyclic. Then by \cite{9}, there exists at least three distinct subgroups of $P$ (hence of $G$) of order $p$. Then these three subgroups together with the trivial subgroup forms the structure $\Gamma_1$ of the Figure \ref{figure 1}. Therefore, by Lemma \ref{lg}, $\Gamma$ is not a line graph. Hence both $P$ and $Q$ are cyclic. \\
    Now, suppose that one of $P$ and $Q$ (say $P$) is not normal in $G$. Then there are at least two (indeed three) distinct Sylow $p$-subgroups $P_1$ and $P_2$ of $G$. If the intersection of all Sylow $p$-subgroups are trivial, then we can choose two distinct elements $a$ and $b$ of order $p$ from $P_1$ and $P_2$ respectively. Let $L$ be the subgroup of $G$ generated by $a$ and $b$. Since the order of $G$ is $p^\alpha q^\beta$, by Burnside Theorem $G$ is solvable. Consequently, $L$ is solvable. Therefore, by \cite[Corollary 2.3, p.319]{10}, $\frac{L}{[L,L]}$ is elementary abelian $p$-group and the canonical homomorphism $\eta: L \longrightarrow \frac{L}{[L,L]}$ splits, where $[L,L]$ is the commutator subgroup of $L$.\\
    Assume that $L$ is a $p$-group. If $L$ is abelian, then $L$ contains a subgroup isomorphic to $\mathbb{Z}_p \times \mathbb{Z}_p$, as $a$ and $b$ are from the distinct Sylow $p$-subgroups of $G$. By Theorems \ref{lcg} and \ref{thm 1}, $\Gamma$ is not a line graph.\\
    If $L$ is not-abelian, then by \cite{9}, there exist three distinct subgroups of order $p$. These subgroups together with trivial subgroup forms the structure of $\Gamma_1$ of Figure \ref{figure 1}. Hence, $\Gamma$ is not a line graph.\\
    This shows that $L$ is not a $p$-group. Let a prime $r (\neq p)$ divide the order of $L$. Then there exists an element $c$ in $L$ of order $r$. If the subgroup $\langle c \rangle$ of $L$ is unique in $L$, then it is normal in $L$. Therefore, we get a non-abelian subgroup of $L$ (hence of $G$) of order $pr$. By Proposition \ref{prop 1}, $\Gamma$ is not a line graph. This implies that there exists two distinct subgroups $\langle c \rangle$ and $\langle c' \rangle$ of $L$ of order $r$.\\
    Let $L'$ be the subgroup of $L$ generated by $c$ and $c'.$ Again, by \cite[Corollary 2.3, p.319]{10}, $\frac{L'}{[L',L']}$ is elementary abelian $r$-group. Since the prime $r$ is different from the prime $p$, $L'$ is the proper subgroup if $L$. We apply the similar argument on $L'$ unless we get a subgroup isomorphic to $\mathbb{Z}_t \times \mathbb{Z}_t$, where $t$ is some prime. This process will terminate as a natural number has only finitely many prime divisors. Therefore, by Theorems \ref{lcg} and \ref{thm 1}, $\Gamma$ is not a line graph.\\
    Now, let the intersection of all Sylow $p$-subgroups be $<d>$ for some non-trivial element $d\in G$. Note that $<d>$ is the characteristic subgroups of $G$. Let the order of $d$ be $p^\gamma$, where $0<\gamma \leq \alpha$. Since there are at least three distinct Sylow $p$-subgroups of $G$, we can choose three distinct cyclic subgroups $K_1$, $K_2$ and $K_3$ of order $p^{\gamma+1}$ containing $<d>$  from distinct three Sylow $p$-subgroups of $G$. Then the subgraph of $\Gamma$ induced by the set $\{<d>,K_1,K_2,K_3\}$ is isomorphic to $\Gamma_1$ of Figure \ref{figure 1}. Hence $\Gamma$ is not a line graph. This implies that both $P$ and $Q$ are normal in $G$. Therefore, $G$ is abelian group. By Theorems \ref{lcg} and \ref{thm 1}, $G$ is a line graph if and only if $G$ is cyclic group of prime power order or a cyclic group of order $pq$.\\
    Next, assume that one of $\alpha$ or $\beta$ is zero. Let $\beta = 0$. Then $G$ is $p$-group. If $G$ is not abelian, then  by  \cite{9}, there exists three subgroups of order $p$. These subgroups together with trivial subgroup forms $\Gamma_1$ of Figure \ref{figure 1}. Hence, $\Gamma$ is not a line graph. Therefore, $G$ is abelian. Hence by Theorems \ref{lcg} and \ref{thm 1}, $G$ is a line graph if and only if $G$ is cyclic group of prime power order.
\end{proof}

	\noindent \textbf{Acknowledgment}: The first author is supported by junior research fellowship of CSIR, India.\\
	


    \end{document}